\documentclass[a4paper,11pt]{article}

\usepackage{comment}
\usepackage{a4wide}
\usepackage[english]{babel}
\usepackage[utf8]{inputenc}
\usepackage[fleqn]{amsmath}
\usepackage{amsfonts}
\usepackage{amsthm}
\usepackage{bbm}
\usepackage{dsfont}
\usepackage{xcolor}
\usepackage{url}

\newtheorem{theorem}{Theorem}
\newtheorem{definition}{Definition}
\newtheorem{proposition}{Proposition}
\newtheorem{lemma}{Lemma}
\newtheorem{fact}{Fact}

\newtheorem{corollary}{Corollary}
\newtheorem*{conjecture}{Conjecture}

\newcommand{\E}{{\mathds{E}}}
\newcommand{\prob}{\mathds{P}}
\newcommand{\I}{{I}}

\newcommand{\corch}[1]{\left[ #1 \right]} 
\newcommand{\pare}[1]{\left( #1 \right)} 
\newcommand{\size}[1]{\left| #1 \right|} 
\newcommand{\abs}[1]{\left| #1 \right|} 
\newcommand{\set}[1]{{\left\{ #1 \right\}}} 
\newcommand{\Po}[1]{ {\rm{Po} }( #1 )} 

\newcommand{\convdist}{\xrightarrow[]{(d)}}

\newcommand{\edge}[1]{ {\rm{edges} }  \pare{#1} } %

\def\X{\mathbb{X}}

\def\Q{\mathbb{Q}}
\def\R{\mathbb{R}}
\def\N{\mathbb{N}}

\def\mx{\mu_{\X}}

\def\l{\lambda}
\def\e{\varepsilon}

\def\ind{\mathbbm{1}}
\allowdisplaybreaks

\title{Poisson generic sequences}

\author{
\begin{tabular}{ccc}
Nicolás Álvarez &
        Verónica Becher &
        Martín Mereb
\end{tabular}
}

\date{\today}

\begin{document}

\maketitle

\begin{abstract}
    Years ago, Zeev Rudnick defined the  Poisson generic real numbers 
by counting the number of occurrences of long  blocks 
of digits in the initial segments of the expansions of the real numbers in a fixed integer base.
Peres and Weiss proved that  almost all real numbers, 
with respect to the Lebesgue measure,
 are Poisson generic, but they did not publish their proof.
In this note, we first transcribe Peres and Weiss' proof and  then
 we show that there are computable Poisson generic instances 
and that  all Martin-L\"of random real numbers are Poisson generic.
\end{abstract}

\section{Introduction and statement of results}

Years ago Zeev Rudnick  defined the {\em Poisson generic} real numbers motivated by his result  in~\cite{Rudnick}  that in almost all dilates of lacunary sequences the number of elements in a random interval of the size of the mean spacing  follows the Poisson law.
By considering a variation on this,
Rudnick defined the notion of Poisson genericity for real numbers 
by counting   the number of occurrences of long  blocks 
of digits in the initial segments of the fractional expansions of the real numbers in a fixed integer base\footnote{He called  the notion  {\em supernormality}. Personal communication from Z. Rudnick to V. Becher, 24 May 2017.}.

Since Rudnick's definition considers just a single integer base, it boils down to  counting occurrences of blocks of symbols in initial segments of infinite sequences of symbols in a given finite alphabet.
Let  $\Omega$ be an alphabet of $b$ symbols, for $b\geq 2$. For each positive integer $k$, let  $\Omega^k$ be the set of words of length $k$ over alphabet~$\Omega$
and let  $\Omega^{\mathbb N}$ be the  set of infinite sequences of symbols in this given alphabet. 
For each~$k$, the initial segment of length~$N+k-1$
of an element in $\Omega^\N$ can be seen as~$N$  almost independent events of words of length~$k$,
each one with equal probability~$p= b^{-k}$.
The expected proportion of the $b^k$ many words  that occur exactly~$i$ times, for each $i=0,1,\ldots$, is 
\[
\binom{N}{i} p^i  (1-p)^{N-i}.
\]
  The Poisson distribution arises as a limit of the binomial distributions as follows, see also~\cite[Page 1]{last}.
  When ${N}p$ is a fixed constant~$\lambda$, for  $i=0,1,\ldots$,
  \begin{align*}
 \lim_{\substack{N\to\infty\\ \lambda=Np}}  \binom{N}{i}  p^i (1-p)^{N-i} 
  & = \lim_{N\to\infty}     \frac{N(N-1) \cdots (N-i+1)}{N^i}
    \left(1- p\right)^{N} \frac{\lambda^i}{i!}
 = e^{-\lambda} \frac{\lambda^i}{i!}
\end{align*}
  
We number the positions in words and  infinite sequences starting from~$1$ and  we write $\omega[l,r]$ for the subsequence of~$\omega$  that begins in position~$l$ and ends in position~$r$. We use interval notation, with a square bracket  when the set of integers includes the endpoint and a parenthesis to indicate that the endpoint is  not included.
For a word~$\omega$ we denote its length as~$|\omega|$. 

For  $j\in\N$, $x\in \Omega^{\N}$,  $k\in\N$ and $\omega \in \Omega^{k}$,   we  write $I_j(x,\omega)$ for  the indicator function  that the word $\omega$ occurs in the sequence  $x$ at position~$j$,

\[
I_j(x,\omega)= \ind_{ \{x[j, j+k) = \omega\} }.
\]
For~a positive real $\lambda$, $x\in\Omega^{\mathbb N}$,  
 $k\in\mathbb N$ and $i\in\mathbb N_0$  we write $Z^\lambda_{i,k}(x)$ for the proportion of words of length~$k$ that occur exactly~$i$ times in $x[1,\lfloor\lambda b^k\rfloor+k)$,

\[\displaystyle
  Z^\lambda_{i,k}(x) = 
  \frac{1}{b^k}
  \#\left\{\omega \in \Omega^k:\sum_{1\leq j\leq  \lambda b^k} I_j(x,\omega) =i \right\}.
\]  

\begin{definition}[Zeev Rudnick]\label{def:Poisson}
Let $\lambda$ be a positive real number. An element $x\in\Omega^{\mathbb N}$ is $\lambda$-Poisson generic if  for every  $i\in \mathbb{N}_0$,

\[
\lim_{k\rightarrow\infty}   Z^\lambda_{i,k}(x)
= e^{-\lambda} \frac{\lambda^i}{i!}.
\]
 An element $x\in\Omega^{\mathbb N}$ is  Poisson generic if it is  $\lambda$-Poisson generic
 for all positive real numbers~$\lambda$.
\end{definition}

Yuval Peres and Benjamin Weiss~\cite{weiss2020} strengthened the definition of Poisson genericity 
by considering  {\em all} sets of positions definable from Borel sets  instead of just sets of positions given by initial segments.\footnote{Talk by Benjamin Weiss entitled  ``Random-like behavior in deterministic systems'',
at Institute for Advanced Study Princeton University USA, June 16 2010.} 

We regard $\Omega$ as a finite probability space with uniform  measure that we denote $\mu$.
For each $x\in \Omega^\N$
and for each $k\in \mathbb N$, on  the product space  $\Omega^k$ with product measure $\mu^k$,
define the integer-valued random measure  
$M^x_k =M_k^{x}(\omega)$ on the  real half-line $\R^+=[0,+\infty)$ by setting for all Borel sets $S\subseteq \R^+$,
\[
M_k^{x}(S)(\omega) = \sum_{j \in \N \cap b^k S} I_j(x,\omega),
\]
where 
$ \N\cap b^k S$  denotes the set of integer values in  
$\{ b^k s: s\in S\}$.

A  point process $Y(\cdot)$ on  $\R^+$ is an integer-valued random measure.
Therefore, $M_k^{x}(\cdot)$ is a point process on  $\R^+$ for each $k\geq 1.$
A Poisson point process on~$\R^+$ 
is a point process $Y(\cdot)$ on $\R^+$  such that the following two conditions hold:
(a) 
for all disjoint Borel sets $S_1, \ldots, S_m$ included in $\R^+$, the   random variables $Y(S_1),\ldots , Y(S_m)$ are mutually independent; and (b)  for each bounded Borel set $S\subseteq \R^+$, $Y(S)$  has the distribution of a  Poisson random variable with parameter equal to the Lebesgue measure of~$S$.
A sequence $\pare{Y_k(\cdot) }_{k\geq 1}$ of point processes converges in distribution to a point process $Y(\cdot)$  if, for every Borel set $S$, the random variables $Y_k(S)$ converge in distribution to $Y(S)$ as $k$ goes to infinity.
     A thorough presentation  on Poisson point processes can be read from~\cite{kingman} or~\cite{last}.


We write $\mu^\N$ for the product  measure on $\Omega^\N$.

\begin{theorem}[Peres and Weiss \cite{weiss2020}]\label{thm:weiss}
For almost all  $x\in \Omega^{\mathbb N}$ with respect to 
the product measure~$\mu^\N,$
the point processes $M_k^{x}(\cdot)$ converge in distribution to a 
Poisson point process on  $\R^+$ as $k$ goes to infinity.\end{theorem}

Peres and Weiss communicated the proof in~\cite{weiss2020} but they did not publish it. The first contribution in this note is a transcription of their proof.

The definition of Poisson genericity, Definition~\ref{def:Poisson},  uses the function~$Z_{i,k}^{\lambda}(x)$,  which can be formulated  in terms of 
$M^{x}_k(S)$ for the sets  $S=(0,\lambda]$, as follows:
\begin{align*}
Z_{i,k}^{\lambda}(x)
&=\mu^k\pare{\omega\in\Omega^k: M^{x}_k((0,\lambda])(\omega)=i}.
\end{align*}
This yields the following corollary of Theorem~\ref{thm:weiss}:

\begin{corollary}[Peres and Weiss \cite{weiss2020}]\label{thm:metric}
Almost all  elements in $\Omega^{\mathbb N}$,  with respect to 
the product measure~$\mu^\N,$
are Poisson generic.
\end{corollary}

Peres and Weiss~\cite{weiss2020} also proved
that  for any fixed positive $\lambda$,  $\lambda$-Poisson genericity implies Borel normality and 
that  the two notions are not equivalent,  witnessed by the fact that Champernowne's  sequence is not $\lambda$-Poisson generic for $\lambda=1$. Their proof method was used in~\cite{Kamae2015,Kamae2018} for other  randomness notions.

The second contribution of this note is  an existence  proof of   {\em computable} Poisson generic elements in~$\Omega^\N$. 
The theory of computability defines the computable functions from  $\N$ to $\N$  and they correspond exactly  to the functions that  can be calculated by an algorithm.  
The notion of computability extends immediately to countable spaces (by fixing an enumeration) and to other objects and spaces, for a monograph on this see~\cite{W}.
An element $x\in \Omega^{\N}$ is  \emph{computable} if there is a computable function $f:\N\to \Omega$ such that  $f(n)$ is  the $n$-th symbol of~$x$. 
We show:

\begin{theorem}\label{thm:computable}
There are  countably many computable Poisson generic elements in $\Omega^\N$.
\end{theorem}

Theorem~\ref{thm:computable} is for Poisson genericity 
as the computable version of Sierpi\'nski's construction~\cite{BF}
or  Turing's algorithm~\cite{BFP,Turing:1937} is for Borel absolute normality (normality to all integer bases).
We follow the same strategy 
first used by Turing
but in the general  form presented in~\cite{Rojas}.
From Theorem~\ref{thm:computable} follows that there are Poisson generic sequences in every Turing degree. To see this, consider  a computable Poisson generic sequence~$x$ and any  given sequence~$y$, and construct a    sequence~$z$  by inserting in~$x$ the symbols of~$y$  at prescribed very widely spaced positions. The set of these   positions should  be computable and should have density zero.

Although almost all elements in $\Omega^\N$
are Poisson generic and there are computable instances,  no explicit example is known.
The recent  work~\cite{bs} gives a construction of explicit  $\lambda$-Poisson generic sequences in an alphabet with at least three symbols, for any positive fixed real number~$\lambda$.

After gathering  statistics on several sequences we arrived to the following.

\begin{conjecture}
The sequences obtained by concatenating the Fibonacci numbers (in any base), the Rudin--Shapiro along squares and the Thue--Morse, along squares, are $1$-Poisson generic.
\end{conjecture}
The automatic sequences  Rudin--Shapiro and Thue--Morse along squares are known to be Borel  normal~\cite{MR,Mullner17}.

The  last result of this note relates  Poisson genericity  with  the notion of  randomness given by the theory of computability  called {\em Martin-L\"of randomness}. A thorough presentation of this notion can be read from~\cite{nies}.

Assume the alphabet $\Omega$  has  $b$ symbols, $b\geq 2$. 
We write $\Omega^{<\N}$ for the set of all finite words~$\bigcup_{k\geq 1} \Omega^k$.
In  the space $\Omega^\N$ with the product measure~$\mu^\N$
consider the basic open sets $B_\omega=\{\omega z: z\in \Omega^{\N}\}$, for each  $\omega\in \Omega^{<\N}$. Then, $\mu^\N(B_\omega)=b^{-|\omega|}$.
A set $O\subseteq \Omega^\N$ is computably  open if $O= \bigcup_{i\geq 1}{B_{f(i)}}$  for some computable function $f:\N\to\Omega^{<N}$.
A sequence $(O_n)_{n\geq 1}$  of open sets  is uniformly computable if there is a  computable function $f:\N\times\N\to \Omega^{<\mathbb N}$ such that 
for each $n\in\N$, $O_n=\bigcup_{i\geq 1}  B_{f(n,i)}$.

A {\em Martin-L\"of test} is a uniformly computable sequence  $(O_{n})_{n\geq 1}$
of  open sets whose measure is computably bounded and goes to $0$ as $n$ goes to infinity.
A~sequence~$x\in\Omega^\mathbb N$ is 
{\em Martin-L\"of random} if, for every 
Martin-L\"of test $(O_{n})_{n\geq 1}$, the sequence
$x$  is not in $\cap_{n\geq 1} O_n$.
Since there are only countably many tests
 it follows that almost all  elements in $\Omega^\N$  are  Martin-L\"of random.

An equivalent formulation says that 
$x\in\Omega^\mathbb N$  is Martin-L\"of random if $x$ is the base-$b$ expansion of a real number $y$ such that  the sequence $(b^n y)_{n\geq 1}$ is uniformly distributed modulo one for all  computably open sets, not just for intervals~\cite{BG}.  Since changing the base representation is achievable by a computable function,  this formulation of Martin-L\"of randomness can be stated requiring that the sequence    $(c^n y)_{n\geq 1}$ be uniformly distributed modulo one for computably open sets, with {\em any}  integer $c\geq 2$.
Here we prove:

\begin{theorem}\label{thm:random}
All Martin-L\"of random elements in $\Omega^\N$  are Poisson generic.
\end{theorem}

Theorem~\ref{thm:weiss} proves a metric result on  a notion of Poisson genericity stronger  than that of Definition~\ref{def:Poisson}
by considering point processes on~$\R^+$.
The technique used to prove Theorems~\ref{thm:computable} and~\ref{thm:random} applies for this stronger notion  as well, after some tweaking in the bounds.

\section{Proof of Theorem~\ref{thm:weiss}}

We follow Peres and Weiss' proof~\cite{weiss2020}. They first give   a randomized result
where one randomizes the sequence~$x\in \Omega^{\mathbb N}$. They call  it the 
{\em annealed} result. 
Then, they obtain the wanted  pointwise  result required in Theorem~\ref{thm:weiss} --also referred as the   {\em quenched} result-- by applying a concentration  inequality.

\subsection{The annealed result}

For each $k\in \N$,  on the product space  $(\Omega^{\N}\times\Omega^k,\mu^{\N}\times\mu^{k})$ 
we define  the integer-valued 
 random measure  $M_k=M_k(x,\omega)$ on  $\R^+$ by
\[
M_k(S)(x,\omega) = \sum_{j \in \N \cap b^k S} I_j(x,\omega),
\]
where $ \N\cap b^k S$  denotes the set of integer values in  
$\{ b^k s: s\in S\}$.

We  write  $A \convdist B$ to indicate convergence in distribution.

\begin{lemma} \label{lemma:annealed}
Let $Y(\cdot)$ be a 
Poisson process on~$\R^+$. Then,  $M_k(\cdot) \convdist Y(\cdot), \text { as } k \to \infty$.
\end{lemma}

The proof of Lemma~\ref{lemma:annealed} uses a well-known sufficient condition
for  a sequence of  point processes to converge to a  Poisson point process.

 \begin{proposition} [cf. {{\cite[Theorem 4.18]{kallenberg2017random}}}]  \label{kallenberg}
Let $(X_k(\cdot))_{k\in\N}$ be a sequence of point processes on~$\R^+$ and
let~$Y(\cdot)$ be a 
Poisson process on~$\R^+$.
If for any $S\subseteq \R^+$ that is a  finite union  of disjoint intervals with rational
endpoints we have \samepage
\begin{enumerate}
    \item $\limsup\limits_{k \to \infty}  \E[X_k(S)] \leq \E[Y(S)]$  and 
    \item $\lim\limits_{k \to \infty}  \prob\pare{  X_k(S) = 0 } = \prob\pare{  Y(S) = 0 }$
\end{enumerate}
then $X_k(\cdot) \convdist Y(\cdot)$, as $k \to \infty$. 
\end{proposition}

The total variation distance $d_{TV}$ between two probability measures $P$ and $Q$ on a 
$\sigma-$algebra~$\mathcal{F}$ is defined via
\begin{align*}
d_{TV}(P,Q)=\sup_{ A\in \mathcal{F}}\left|P(A)-Q(A)\right|.
\end{align*}
For a random variable $X$ taking values in $\R$, 
the distribution of~$X$ 
is the probability measure~$\mu_X$ on~$\R$ defined as the push-forward of the probability measure on the sample space of~$X$. The total variation distance between two random variables~$X$ and $Y$ is simply 
\begin{align*}
d_{TV}(X,Y) = d_{TV}(\mu_X, \mu_Y).
\end{align*}
Notice that $X$ and $Y$ do not need to be defined over the same space.

Given a family $\set{ I_j}_{j \in J}$ of random variables on the same probability space, 
a \emph{dependency graph} for such a family is a graph $L$
with underlying vertex set $J$
 such that for any pair of disjoint subsets $A,B\subseteq J$  of vertices  
 with no  edge $e = (a,b), \, a\in A,\, b\in B$ connecting them,  the subfamilies
$\set{ I_i}_{i \in A}$ and $\set{ I_j}_{j \in B}$ are mutually independent.

\begin{proposition}[{{\cite[Theorem 6.23]{janson2011random}}}] \label{prop:dtv}
Let $\Po{\lambda}$  be  a Poisson random variable with mean $\lambda$.
Let  $\set{ I_j}_{j \in J}$ be a family  of random variables on a given  probability space and
let $L$ be its  dependency graph  with underlying vertex set $J$.
Suppose that the random variable $X_J = \sum_{j\in J} I_j$
satisfies $\l = \E \corch{ X_J} = \sum_{j\in J} \E \corch{ I_j } $.
Then,
\begin{align*}
    d_{TV} (X_J, \Po{\l} )
    \leq \min\set{ 1, \l^{-1}}\pare{  \sum_{j \in J} \E \corch{I_j}^2 
    +  \sum_{\substack{ i,j : (i,j) \in \edge{L}  }}
\Big(     \E \corch{I_i I_j} + \E \corch{I_i}  \E \corch{ I_j } 
    \Big)}.
\end{align*}
\end{proposition}
For a measurable set $S\subseteq \R^+$, 
we write $|S|$ for the Lebesgue measure of~$S$.
\begin{proof}[Proof of Lemma~\ref{lemma:annealed}]
We apply Proposition~\ref{kallenberg}.
For the first condition,
it is enough to consider $S\subseteq \mathbb R^+$ to be an interval $(p,q)$ with rational endpoints, in which case
\begin{align*} 
\E \corch{M_k(S)} &= \int\limits_{(x,\omega)\in \Omega^\N\times\Omega^k}
M_k(S)(x,\omega)\  {\rm d}(\mu^{\N}\times\mu^{k}) \\
    &= \frac{1}{b^k} \sum_{\omega \in \Omega^k} 
    \sum_{j \in \N \cap b^k S} \int\limits_{x \in \Omega^\N} I_j(x, \omega) {\rm d} (\mu^\N) \\
    & = \frac{1}{b^{2k}}
    b^k \pare{b^k |S| + {\rm O}(1)}.
\end{align*}
Then, $\E \corch{M_k(S)}$ converges to $|S|$  as $k$ goes to $\infty$.
The ${\rm O}(1)$ term is in fact bounded by~$2.$

For the second  condition of Proposition~\ref{kallenberg}
we show that 
when   $S$ is finite union of intervals with rational endpoints,
the total variation distance 
$d_{TV}(M_k(S), Y(S))$ goes to $ 0$ as $k $ goes to infinity. 
This implies that  the sequence $(M_k(S))_{k\geq 1}$ of random variables  converges in distribution to the  Poisson random variable~$Y(S)$.

We regard  the indicator functions $I_j=I_j(x,\omega)$ as random variables
on  the 
space \linebreak
$(\Omega^\N\times\Omega^k,\mu^\N\times\mu^k)$, 
\[
I_j(x,\omega) = \ind_\set{x[j, j+k) = \omega}.
\]
The dependency  of these random variables is very sparse. 
There is some 
dependence between $I_i$ and $I_j$ only when 
$ \abs{j - i} < k$. Even in such a case,
$I_i(x,\omega) I_{j}(x,\omega) = 1$ 
is only possible when 
 the prefix of $\omega$ of length
 $k-(j-i)$ 
  is the same as the suffix 
of the same length. 
If $i<j$ and $ {j - i} < k$ then 
\[
\mu^k \Big(\omega\in\Omega^k:  \omega(j-i, k] =\omega[1, k-(j-i)]\Big)
= b^{-k + (j-i)}
\] and for each of these $\omega$'s
\[
\mu^\N \Big(x\in \Omega^N_{}: x[i, i+k) = x[j, j+k) = \omega \Big)
=b^{-k - (j-i)}.
\]
Hence, 
\[
\mu^\N\times\mu^k\pare{(x,\omega)\in\Omega^\N\times\Omega^k: I_i(x,\omega)I_j(x,\omega)=1} = b^{-2k},
\]
which is the same as if $I_i$ and $I_j$ were independent.
Notice that 
$\E \corch{\I_j} = b^{-k},
$ and 
$\E \corch{\I_{i} \I_{j}}  = b^{-2k}$. 
The dependency graph $L$ is:
$(i,j) \in \edge{L}$ if and only if $|i - j| < k$.
We apply Proposition~\ref{prop:dtv} to bound $d_{TV}(M_k(S), Y(S) )$, 
where $Y(S)$ has a Poisson distribution with mean~$|S|.$
For a union of $n$ disjoint intervals $S = \bigcup\limits_{i=1}^n (p_i, q_i)$
it yields, 
\begin{align*}
d_{TV}\pare{ M_k(S),  Y(S) }  
&\leq \min\set{ 1, |S|^{-1}} 
\left( \sum_{j \in \N \cap b^k S} {\E [\I_j]}^2 
    + \sum_{\substack{i,j\in \N \cap b^k S \\ |i - j| < k }}
    \Big( \E[\I_i \I_j] + \E [\I_i]\E [\I_j]\Big)
    \right) \\
    & \leq  \sum_{ j \in \N \cap b^k S} b^{-2k} + \sum_{\substack{i,j\in \N \cap b^k S \\ |i - j| < k }} 2 b^{-2k} \\
    & \leq  \pare{ |S| b^k + n } b^{-2k} + \pare{ |S| b^k + n } \ 2k \ 2b^{-2k}.
\end{align*}
The last expression goes to $0$ as~$k$ goes to infinity.
Then, $M_k(S) \convdist Y(S)$,
as~$k$ goes to infinity.
\end{proof}

\subsection{The quenched result}

We use now a classical  concentration  inequality, which estimates the error from the average behaviour.

\begin{proposition}[McDiarmid's inequality \cite{McDiarmid1989}] \label{prop:McD} 
Let $X_1, \ldots, X_N$ be independent random variables taking values in some set $\Omega.$ 
Assume $f : \Omega^N \to \R$ satisfies that for any two 
vectors $x, x'\in \Omega^N$  which differ only in a single coordinate, we have 
\begin{align}\label{e:inequhyptheom}
    \abs{f (x) - f (x')} \leq c,\tag{$\dagger$}
\end{align}
 for some positive $c=c(N).$
Let us write $f(X)$ for the composition $f(X_1,\ldots,X_N)$ and let $\prob$ denote the probability on the underlying domain.
Then, for any $t \geq 0,$ we have 
\begin{align*}
    \prob\pare{ \abs{f(X) - \E [f(X)] } > t } \leq 2\exp\pare{ \frac{-2t^2}{Nc^2} }.
\end{align*}
\end{proposition}

We can now give the proof of Theorem \ref{thm:weiss}. We use the well-known Borel--Cantelli lemma, see  \cite[Chapter 3, Lemma 1]{feller}, which says that for a sequence
of subsets  $(A_n)_{n\in\N}$  in a probability space $(\X,\mu_\X)$,
if $\sum_{n\geq 1}\mu_\X(A_n)<\infty$, then $\mu_\X(\limsup A_n)=0$, that is, the set of points which are contained in infinitely many $A_n$ has null  measure.
Under these conditions, $\X - \limsup A_n$ is a full measure set.

\begin{proof}[Proof of Theorem \ref{thm:weiss}]
We want to show that, for almost every $x$ 
in $\Omega^\N$, as $k$ goes to infinity,
the processes $M^{x}_k(.)$ converge in distribution to 
$Y(\cdot)$, 
where $Y(\cdot)$ is a 
Poisson process on $\R^+$. 
By Proposition~\ref{kallenberg} it suffices to consider sets 
$S\subseteq \R^+$ that are finite unions of disjoint intervals with rational endpoints.
The first condition of Proposition~\ref{kallenberg} holds because  or each such $S$,
$\E[M_k^{x}(S)]=\size{S} + {\rm O}(b^{-k}).$

We now  verify the second condition of  Proposition~\ref{kallenberg}. Let $n$ be the number of disjoint intervals of~$S$.
Given $x \in \Omega^\N ,$ the probability
$\mu^k\pare{\omega\,:\,M^{x}_k(\omega)(S)=i }$ 
depends on $N = |S| b^k + \e n k $ coordinates of $x$,
for some $\e \in [0,1).$ This is because $S$ is the union of~$n$ disjoint intervals and for each of them one must consider at most one extra coordinate to 
take into account its alignment with integer values, and $k-1$ extra coordinates to fit $\omega$. 
We apply Proposition~\ref{prop:McD} 
to the function
$  f_k  : \Omega^N \to \R  $
given  by
\[ 
f_k(x)= \mu^k\pare{\omega\,:\,M^{x}_k(\omega)(S)=i }. 
\]
Since a one-coordinate change in $x$ affects no more than $k$ of the $\omega$'s in the counting
for $M^{x}_k(\omega),$ the inequality~\eqref{e:inequhyptheom} is satisfied 
with $ c = {k}{b^{-k}}.$
By choosing $t_k = 1/k$ one gets
\begin{align*}
    \sum_{k = 1}^\infty \mu^{\N}\pare{x: \,\abs{f_k(x) - \E [f_k(x)]} > t_k } \leq 
     2 \sum_{k = 1}^\infty \exp\pare{  { - k^{-4} b^k \pare{ |S| + {2nk}{b^{-k}} }^{-1}  } }
\end{align*}
and this expression converges.
Then, by the Borel--Cantelli lemma
the limsup event
\begin{align*} 
\set{\,  x: \, \abs{f_{k}(x) - \E [f_{k}(x)]} > t_{k} \, \text{for infinitely many }k } 
\end{align*}
has probability $\mu^\N$ zero. That is to say, for almost every $x\in\Omega^\N$
the probabilities 
\[
\mu^k\pare{\omega\,:\,M^{x}_k(S)(\omega)=i }
\]
converge, as $k$ goes to infinity, to the same limit as that of 
\[\E \corch{  \mu^k\pare{\omega\,:\,M^{x}_k(S)(\omega)=i
}}.
\]
Given the identity 
\[ 
\E \corch{  \mu^k\pare{\omega\,:\,M^{x}_k(S)(\omega)=i
}}= \mu^\N \times  \mu^k\pare{(x,\omega)\,:\,M_k(S)(x,\omega)=i } .
\]
and that, by  Lemma~\ref{lemma:annealed}, 
\[
M_k(S)\convdist Y(S), \text{ as $k\to\infty$},
\]
 we conclude that  
\[
\mu^k\pare{\omega\,:\,M^{x}_k(S)(\omega)=i }
\text{converge  to } \prob\pare{Y(S)=i}, 
\text{ as $k\to\infty$}.
\]
This happens for every $i\geq 0$ and
for every $S$  that is a finite union of intervals with rational endpoints. 
Since a countable union of sets of probability zero has probability zero as well, we conclude that
for $\mu^\N$-almost every $x\in\Omega^\N$, 
\[
M^{x}_k(S)\convdist Y(S), \text{ as $k\to\infty$}
\]
 for all such sets $S$.
\end{proof}

\section{Proofs of Theorem \ref{thm:computable} and Theorem \ref{thm:random}}

In this section we use three technical results from~\cite{Rojas}  for {\em computable metric spaces}~$\X$  and    {\em computable probability measures}~$\mx$ on~$\X$.   We start with the primary definitions.

The notion of computability is defined for many objects and spaces~\cite{W}.  For instance, a real number $x$ is  computable 
if there is a computable function  $f:\N\to \Q$ such  that $|x-f(n)|\leq 2^{-n}$, for all $n$.
A sequence of elements in a space~$\X$ is {\em uniformly computable} if there is a computable function $f: \N\times\N\to\X$
such that the $n$-th element in the sequence is computed by the projection  $f_n(x)= f(n,x)$.

A metric space  is a pair $(\X,d)$, where $\X$ is non-empty and $d$ is a distance between  elements in $\X$. 
A metric space is  complete if every Cauchy sequence of elements in $\X$ has a limit also  in~$\X$.
A space $\X$ is separable if it contains a countable dense subset.
A  computable metric space is a triple $(\X, d, S)$, where
$X$ is a separable  metric space (also known as a Polish space)  that contains a
countable dense subset 
$S = \{s_i\in \X : i \in N\}$
and the distance $d(x,y)$ between elements $x,y$ in $S$ is computable.
A probability measure $\mu_{\X}$  over a computable metric space $(\X, d,S)$   is  computable if  the probability measure of any  finite union of  balls with rational radius and centered in elements in $S$ can be computably approximated from below, uniformly.

\begin{fact}
The space $(\Omega^\N,d,S)$ 
where
 $S$ is the set of computable elements in~$\Omega^\N$ and \linebreak
$d(x,y)=b^{-lcp(x,y)}$  
with  $b$ equal to  the cardinality of~$\Omega$
and $lcp(x,y)$ equal to  the length of the longest common prefix between $x$ and $y$, 
 is  a computable complete metric space.
The product measure~$\mu^\N$ is a computable probability measure on the Borel sets of $\Omega^\N$.
\end{fact}

A sequence  $(x_i)_{i\geq 1}$ of real numbers  is {\em effectively summable} if for every~$\varepsilon\in \mathbb Q$,  we can compute $n=n(\varepsilon)$ such that 
$\sum_{i\geq n}x_i<\varepsilon.$
A sequence $(U_n )_{n\geq 1}$ of open sets included in a computable metric space $\X$
is {\em constructive Borel--Cantelli }
if it is a uniformly computable sequence of
open sets such that the sequence $(\mx(\X \setminus U_n))_{n\geq 1}$ is effectively summable.
Given a constructive Borel--Cantelli  sequence $(U_n )_{n\geq 1}$ the corresponding {\em Borel--Cantelli set} is~$\bigcup_{k\geq 1}\bigcap_{n>k} U_n$.

\begin{lemma}[{{\protect{\cite[Lemma 3]{Rojas}}}}] \label{lemma:transformed}
Let $\X$ be a computable probability space with computable measure~$\mu_X$.
Every constructive Borel--Cantelli sequence can be 
transformed into a constructive
Borel--Cantelli sequence $(U_n)_{n\geq 1}$ giving the same Borel--Cantelli set, with 
$\mx(\X \setminus U_n) < 2^{-n}$.
\end{lemma}
\begin{proof}
Let  $(V_n)_{n\geq 1}$ be  a constructive Borel--Cantelli sequence. 
As $\pare{\mx(\X \setminus V_n)}_{n\geq 1}$ is effectively summable, 
an increasing sequence $(n_i)_{i\geq 0}$
 of integers can be computed such that for all~$i\geq 1$,
$ \sum_{n\geq n_i}\mx(\X \setminus V_n ) < 2^{-i}$.
We now gather the $V_n$ by blocks, setting
\[
 U_i = \bigcap_{n_i\leq n <  n_{i+1}}  V_n.
\] 
Then, the sequence $(U_i)_{i\geq 1}$ of open sets is 
is uniformly computable,
$ \mx(X \setminus U_i)<2^{-i} $  and 
\[
\bigcup_{k\geq 1}\bigcap_{n\geq k} V_n =
\bigcup_{i\geq 1} \bigcap_{n\geq n_i }V_n =
\bigcup_{i\geq 1} \bigcap_{j\geq i }U_j.
\]
\end{proof}

The diameter of a set $V$ in a metric space is the supremum of distances between its elements and it is denoted by ${\rm diam}(V)$.
We write $\overline{V}$ for the closure of $V$.

\begin{lemma}[{{\protect\cite[Lemma 4]{Rojas}}}]\label{lemma:point}
Let $\X$ be a computable  metric space with computable measure~$\mu_X$.
 Let $(V_i)_{i\geq 1}$ be a sequence of 
uniformly computable non-empty open sets 
such that for each $i$, $\overline{V}_{i+1}\subseteq V_i$ and ${\rm diam}(V_i)$ converges effectively to $0$ as $i$ goes to infinity. 
Then $\bigcap_{i\geq 1} V_i$ is a singleton containing a computable element.
\end{lemma}

\begin{proof}
Since each $V_i$ is non-empty there is a computable sequence of elements $(s_i)_{i\geq 1}$, $s_i\in  V_i$. 
This is a Cauchy sequence, which converges by completeness. 
Let $x$ be its limit: it is a computable element as ${\rm diam}(V_i )$ 
converges to $0$ in an effective way. 
Fix some $i$. For all $j\geq i$, $s_j \in V_j \subseteq\overline{ V}_i$, so
$x=\lim_{j\to\infty} s_j \in \overline{V}_i$.
Hence $x\in\bigcap_{i\geq 1} \overline{V}_{i\geq 1} =\bigcap_{i\geq 1} V_i$. 
\end{proof}

\begin{lemma}[{{\cite[Theorem 1]{Rojas}}}]\label{lemma:rojas}
 Let $\X$ be a  computable complete metric space 
  and with  computable  probability measure~$\mx$.
Every constructive Borel--Cantelli set 
    contains a sequence of uniformly computable elements which is dense in the support of~$\mx$.
\end{lemma}

\begin{proof} 
Let $(U_n)_{n\geq 1}$ 
be a constructive Borel--Cantelli sequence such that $\mx(U_n)>1 -2^{-n}$ (by Lemma~\ref{lemma:transformed} this can always be obtained). 
Let $ B$ be a basic open set. 
In $B$ we construct a computable element which lies in $\bigcup_{n\geq 1} \bigcap_{k\geq n} U_k$, 
in a way that is uniform in~$B$. 

Here is the construction. Let $V_0= B$ and $n_0$  be such that $\mx(B) > 2^{-n_0 +1}$ 
(such an $n_0$ can be effectively found from $B$).
We construct a sequence $(V_i)_{i\geq 1}$ of uniformly 
computable
open sets and a computable increasing sequence $(n_i)_{i\geq 1}$ of positive integers satisfying:

\begin{itemize}\samepage
\item[ (1)] $\mx(V_i) + \mx\Big( \bigcap\limits_{k\geq n_i} U_k \Big) > 1$, 
\item[ (2)] $V_i \subseteq  \bigcap\limits_{n_0\leq k<n_i} U_k$, 
\item[ (3)] ${\rm diam}(V_i) \leq  2^{-i+1}$,
\item[ (4)] $\overline{V}_{i+1}\subseteq  V_i$.
\end{itemize}
The last two conditions assure that $\bigcap_{i\geq 1} V_i$ is a computable element,
 the second condition assures that this element
lies in $\bigcap_{k\geq n_0} U_k$.
Suppose $V_i$ and $n_i$ have been constructed.

 By the first condition, 
\[\mx\Big(V_i \cap \bigcap_{k\geq n_i} U_k\Big) > 0,\]
so there exists a basic open set $B' $ of radius $2^{-i-1}$  such that 
\[
\mx\Big(V_i \cap\bigcap_{k\geq n_i} U_k \cap  B'\Big) > 0.
\] 
Then, there is $m > n_i$ such that  
\[
\mx\Big(V_i \cap\bigcap_{k\geq n_i} U_k \cap  B'\Big)  > 2^{-m+1},
\]
and hence, 
\[
\mx\Big(V_i \cap \bigcap_{n_i\leq k<m} U_k \cap B'\Big) > 2^{-m+1}.
\]
this  inequality  can be semi-decided, such an $m$ and a $B'$ can be effectively found. 
For $ V_{i+1}$ , take any finite union of basic open sets whose closure is contained in 
\[
V_i \cap\bigcap_{n_i\leq k<m }U_k \cap B'
\]
and whose measure is greater than $2^{-m+1}$ . 
Put $n_{i+1} = m$. 
Conditions 2, 3 and 4 directly follow from the construction,
condition 1 follows from 
\[
\mx\Big(V_{i+1}\Big) > 2^{-m+1} > 1-\mx\Big(\bigcap_{k\geq m}U_k\Big).\]
\end{proof}

Tail estimates quantify the rate  of decrease of probabilities away from the central part of the distribution. 
As a corollary of the proof of Theorem~\ref{thm:weiss} we obtain the following result.
It considers the space $\Omega^\N$, 
the measures $\mu^\N$, $\mu^k$, $\mu^\N\times \mu^k$, for every $k\geq 1$,
and the integer-valued random measures on $\R^+$,
$M_k^{x}=M_k^{x}(\omega)$ just on sets $S=(0,\lambda]$, for $\lambda\in\R^+$.


\begin{lemma}[Tail Bound] \label{lemma:tail}
Let $b\geq 2$ be the number of symbols of alphabet $\Omega$,
$i$ a non-negative integer and $\lambda \in \R^+$. Then, for every $k \geq k_0(\lambda) = \max\set{24, 2\log_2(\lambda+1)}$ 
we have
\[
\mu^{\N} \pare{x \in \Omega^{\N} : \abs{\mu^k\pare{  M_k^{x}((0,\lambda]) = i } - \frac{e^{-\lambda} \lambda^i}{i!} }
> 2/k } \leq \exp\pare{\frac{-2 b^k}{\lambda  k^4}}.
\]  
\end{lemma}
\begin{proof}
Let $\Po{\lambda}$  be  a Poisson random variable with mean $\lambda$.
From the proof of Lemma~\ref{lemma:annealed},
\[
d_{TV} \pare{ M_k((0, \lambda]), \Po{\lambda} } \leq (\lambda+1) b^{-k} 5k ,
\]
which is less than
$1/k$ provided $k\geq k_0(\lambda) = \max \set{ 24, 2 \log (\lambda + 1) }.$ 
\\
This implies, for every~$i\geq 0$,
\[ \abs{ 
\mu^{\N}\times \mu^k\Big(M_k((0,\lambda]) = i\Big)  - \prob\Big(  \Po{\lambda}  = i \Big)  } < 1/k.
\]
Using Proposition~\ref{prop:McD} for the functions $f_k$ introduced in the proof of  Theorem~\ref{thm:weiss}, we know that for every $k\geq 1$ and $i\geq 0$,
\[
\mu^{\N} \pare{ x  \in \Omega^{\N}:
        \abs{
\mu^{k}\Big(M_k^{x}((0,\lambda])  = i\Big) - 
\mu^\N\times\mu^{k}\Big(M_k((0,\lambda]) = i\Big) 
        }  > 1/k
    }
    \leq
   2 \exp\pare{\frac{-2 b^k}{\lambda  k^4}}.
\]
Given that
$\prob\pare{\Po{\lambda}=i}={e^{-\lambda} \lambda^i }/{ i!},$
combining the two inequalities above
we obtain the wanted result.
%
%
%
%
%
%
\end{proof}
\medskip

\begin{proof}[Proofs of Theorems~\ref{thm:computable} and~\ref{thm:random}]
Consider the  topology
generated by the countable family of basic open (and closed)
sets 
$\{ \omega z: z\in\Omega^{\mathbb N}\}$
where $\omega$ varies over $\Omega^{<\mathbb N}$.
For each integer $k\geq 1$, define the open sets
\[
O_k = \bigcup_{\lambda \in L_k} \bigcup_{i \in J_k} Bad(\lambda, k, i)
\]
where
\begin{align*}
Bad(\lambda, k, i) &= 
\set{ x \in \Omega^{\mathbb{N}} : \abs{\mu^k\pare{ M_k^{x}((0,\lambda]) = i } - \frac{e^{-\lambda} \lambda^i}{i!}} > 2/k },
\\
L_k &= \set {p/q: q \in \set{1, \ldots , k}, p/q < k},
\\
J_k &= \set{0 ,\ldots , b^k-1}.
\end{align*}
Using  Lemma~\ref{lemma:tail} we  give an upper bound of  $\mu^\N(O_k)$.
For each  $k\geq 24$,
\begin{align*}
\mu^\N(O_k) &= \mu^\N \pare{ \bigcup_{\lambda \in L_k} \bigcup_{i \in J_k} Bad(\lambda, k, i) }\\
 &\leq  \sum_{\lambda \in L_k} \sum_{i=0}^{b^k-1} \mu^\N(Bad(\lambda, k, i)) \\
 &\leq  \sum_{\lambda \in L_k} \sum_{i=0}^{b^k-1} 2\exp \pare{\frac{-2b^k}{\lambda k^4}}\\
 &\leq  \sum_{\lambda \in L_k} \sum_{i=0}^{b^k-1}2 \exp \pare{\frac{-2b^k}{k k^4}}\\
 &= \sum_{\lambda \in L_k} 2b^k \exp \pare{\frac{-2b^k}{k^5}}\\
 &\leq 2 b^k k^3 \exp \pare{\frac{-2b^k}{k^5}}.
\end{align*}
It follows that   $(\mu^\N(O_k))_{k\geq 1}$ is effectively summable.
Notice that for 
 \[
 U_k = \Omega^{\N} \setminus O_k,
 \]
the set $\bigcup_{k\geq 1}\bigcap_{n>k} U_n$ is a Borel--Cantelli set.

 Applying  Lemma~\ref{lemma:rojas}  on the space $\Omega^{\mathbb N}$
 we conclude 
that there is a sequence of uniformly computable elements.
Each of them is $\lambda$-Poisson generic for every rational~$\lambda$.
    To prove that the property holds for all real numbers, observe that for every pair of positive reals  $\lambda,\lambda'$, with~$\lambda < \lambda'$, 
\[ 
M_k^x((0,\lambda'])(\omega) 
- M_k^x((0,\lambda])(\omega) = 
\sum_{j \in \N \cap b^k [\lambda, \lambda')} 
I_j(x,\omega ),
\] 
{ where, since  $x$ is fixed,  $I_j(x,\omega )$ is a function on $\Omega^k$. }
Hence,
\[
d_{TV}(M_k^x((0,\lambda']), M_k^x((0,\lambda]))
\leq \frac{1}{b^k}\#\pare{\N \cap b^k [\lambda, \lambda')  }
= \lambda' - \lambda + {\rm O}(b^{-k}).
\] 
Also observe that  $d_{TV}(\Po{\lambda'}, \Po{\lambda}) \to 0$ as $\lambda \to \lambda'$.
From these two observations and the fact that the rational numbers are a dense subset of the real
numbers we conclude that each  element in $\bigcup_{k\geq 1}\bigcap_{n>k} U_n$ 
is $\lambda$-Poisson
generic for every positive real $\lambda$ and hence, Poisson generic.
 This completes the proof of  Theorem~\ref{thm:computable}.
\medskip

The remaining lines prove Theorem~\ref{thm:random}.
We show that 
all non-Poisson generic elements in~$\Omega^\N$ are not 
Martin-Löf random.
For this we define a Martin-L\"of test $(T_m)_{m\geq 1} $ such that $\bigcap_{m\geq 1} T_m$   contains all the  non Poisson generic elements.
Fix $k_0=24$. Define  $(T_m)_{m\geq 1} $ by 
\[
T_m =\bigcup_{k\geq m +k_0 } O_{k}.
\]
 Clearly $(T_m)_{m\geq 1} $  is a Martin-L\"of test
because it is  a uniformly computable sequence of open sets, 
$\mu^\N(T_m) $ is computably bounded and it  goes to $0$ as $m$ goes to infinity,
\begin{align*}
\mu^\N(T_m) &\leq \sum_{k\geq m+k_0} \mu^\N(O_k)
\leq  \sum_{k\geq m+k_0} 2 b^k k^3 \exp\pare{ \frac{-2b^k}{k^5} }.
\end{align*} 

Now we  prove that for every $m_0$,
$\bigcap_{m\geq m_0}(\Omega^{\mathbb N}\setminus T_m)$ contains only Poisson generic elements. 
 By way of contradiction assume there exists a value  $m_0$ such that   $x\in\bigcap_{m\geq m_0}(\Omega^\N \setminus T_m) $
but~$x$ is not Poisson generic. Using the same argument as above,  $x$ is not $\lambda$-Poisson generic
 for some positive
rational~$\lambda$.
Then, 
there is a non-negative integer~$i$,
a positive real~$\varepsilon$ and  infinitely many values~$k$ such that
\[
\abs{ 
\mu^k\pare{M_k^{x}((0,\lambda]) = i }  - \frac{e^{-\lambda} \lambda^{i}}{i!} } > \varepsilon.
\]
Fix $k_1\geq m_0$ large enough such that 
$\lambda\in L_{k_1}$, 
$i\in J_{k_1}$ 
and $\varepsilon>2/k_1$.
Since 
$(L_k)_{k \geq 1}$ and $(J_{k})_{k \geq 1}$ are increasing and $2/k$ is decreasing in~$k$,
this is still valid for every $k \geq k_1$.
Since we assumed $x\in\bigcap_{m\geq m_0}(\Omega^{\mathbb N}\setminus T_m)$
then, for every $k\geq k_1$ and for every  $i\in J_{k}$ we have
\[
\abs{ 
\mu^k\pare{M_k^{x}((0,\lambda]) = i } - \frac{e^{-\lambda} \lambda^i}{i!} } < 2/k.
\]
Since $i\in J_k$ and  $2/k< \varepsilon$  we reached a contradiction.
Therefore, all  elements  in  $\bigcap_{m\geq 1} (\Omega^{\mathbb N}\setminus T_m)$ are $\lambda$-Poisson generic for every positive rational~$\lambda$, hence Poisson generic.

Finally, consider any     $x\in\Omega^\N$  that is not  Poisson generic. Then, 
$x$ belongs to no set
\[
W_n=\bigcap_{m\geq n} (\Omega^{\mathbb N}\setminus T_m),
\]
for any~$n$.
Thus,  $x$ belongs,  for each $n$, to the  complement set
$(\Omega^{\mathbb N}\setminus W_n)$.
Then,
\[
x\in \bigcap_{n\geq 1} (\Omega^{\mathbb N}\setminus W_n)
=
\bigcap_{n\geq 1}\Big(  \bigcup_{m\geq n} T_m\Big) = \bigcap_{n\geq 1} T_{n}.
\]
Hence, $x$ is not Martin-L\"of random. This completes the proof of Theorem~\ref{thm:random}.
\end{proof}
\bigskip

\noindent{\bf Acknowledgements.}
We thank Benjamin Weiss for allowing us to transcribe his proof of 
Theorem~\ref{thm:weiss} and for his lively comments.
We also thank Zeev Rudnick for having introduced  us in the world of the Poisson generic sequences.
We are grateful to Inés Armendariz and to an anonymous referee for multiple comments that helped us to improve the presentation.

\bibliographystyle{plain}
\bibliography{revised-abm}

\noindent
Nicol\'as Álvarez \\
  ICC CONICET Argentina -  {\tt  nico.alvarez@gmail.com}
\medskip

\noindent
Ver\'onica Becher \\
 Departamento de  Computaci\'on, Facultad de Ciencias Exactas y Naturales \& ICC  \\
 Universidad de Buenos Aires \&  CONICET  Argentina-  {\tt  vbecher@dc.uba.ar}
\medskip

\noindent
Martín Mereb \\
 Departamento de Matemática, Facultad de Ciencias Exactas y Naturales \& IMAS \\
 Universidad de Buenos Aires \&  CONICET Argentina-  {\tt  mmereb@gmail.com}
\end{document}